\documentclass{publmathdeb}

\usepackage{amsmath,amscd,amsthm}
\usepackage{epsfig}
\usepackage{amssymb}
\usepackage{amsfonts,flafter,epsf}
\usepackage{newlfont}
\usepackage{xcolor} %
\usepackage{mathtools}
\usepackage{rotating}
\usepackage{comment}

\graphicspath{{./Figures/}}

\theoremstyle{plain} %
\newtheorem{theorem}{Theorem}
\newtheorem{lemma}[theorem]{Lemma}
\newtheorem*{lemma*}{Lemma} %

\newtheorem{proposition}[theorem]{Proposition}

  \theoremstyle{definition} %
  \newtheorem{definition}[theorem]{Definition}

  \theoremstyle{remark}
  \newtheorem{example}[theorem]{Example}

  \newcommand{\RSet}{{\mathbb{R}}}
  \newcommand{\ZSet}{{\mathbb{Z}}}
  
  \newcommand{\QSet}{{\mathbb{Q}}}
  \newcommand{\CSet}{{\mathbb{C}}}

\newcommand\set[1]{\{\,#1\,\}}

\newcommand{\Sing}{\mathop{\mathrm{Sing}}\nolimits}

\newcommand{\rk}{\mathop{\mathrm{rk}}\nolimits}

\newcommand{\im}{\mathop{\mathrm{im}}\nolimits}

\newcommand{\Hom}{\mathop{\mathrm{Hom}}\nolimits}
\newcommand{\har}{\mathop{\mathrm{char}}\nolimits}
\newcommand{\interior}[1]{\mathop{\mathrm{int}}\nolimits(#1)}

\newcommand{\al}{\alpha}

\newcommand{\w}{\omega}

\renewcommand{\phi}{\varphi}

\renewcommand{\b}{{b_1'(M)}}
\newcommand{\h}{{\cal H}}

\newcommand\phiq{\phi_\QSet}
\newcommand\hq{H_\QSet}

\newcommand{\connsum}{\mathrel\#}
\newcommand{\opConnsum}{\mathop\#}
\newcommand{\opTimes}{\mathop\times}

\newcommand{\from}{{\,:\;}}

\definecolor{foridiots}{RGB}{0,150,0}
\specialcomment{foridiot}{\bgroup\color{foridiots}}{\egroup}
\makeatletter
\@ifundefined{idiot}%
{\newcommand\textforidiot[1]{\leavevmode\unskip}\excludecomment{foridiot}}%
{\newcommand\textforidiot[1]{\textcolor{foridiots}{#1}}}
\makeatother

\begin{document}

\title{Isotropy index for the~connected~sum and~the~direct~product~of~manifolds}
  \author{I. Gelbukh}
    \email{gelbukh@member.ams.org} %
    \urladdr{http://www.i.gelbukh.com} 

\begin{abstract}
A subspace or subgroup is isotropic under a bilinear map if the restriction of the map on it is trivial.
We study maximal isotropic subspaces or subgroups under skew-symmetric maps, and in particular the isotropy index---the maximum dimension of an isotropic subspace or maximum rank of an isotropic subgroup. 
For a smooth closed orientable manifold $M$, we describe the geometric meaning of the isotropic subgroups of the first cohomology group with different coefficients under the cup product. We calculate the corresponding isotropy index, as well as the set of ranks of all maximal isotropic subgroups, for the connected sum and the direct product of manifolds. 
Finally, we study the relationship of the isotropy index with the first Betti number and the co-rank of the fundamental group. 
We also discuss applications of these results to the topology of foliations.
\end{abstract}

\newcommand\sep{, }
  \keywords{%
Isotropic subspace\sep 
cohomology\sep 
cup product.
}

  \maketitle %

\section{Introduction}

Let $M$ be a smooth closed orientable connected $n$-dimensional manifold.
We study \emph{iso\-tro\-pic} subgroups (subspaces) $H$ of its cohomology group (space) $H^1(M;R)$, where $R$ is a field or the ring of integers,
under the cup-product 
\begin{align}
{\smile}\colon H^1(M;R)\times H^1(M;R)\to H^2(M;R),
\label{eq:intro-cupprod}
\end{align}
i.e., $H\subseteq H^1(M;R)$ such that $H\smile H=0$. To avoid duplication of terminology, such as ``rank (dimension),'' we will refer to $H^i(M;R)$ as modules over the ring~$R$.

Specifically, we study the set $\h(M;R)$ of ranks of maximal isotropic submodules of $H^1(M;R)$ and the corresponding \emph{isotropy index} 
\begin{align}
h(M;R)=\max\h(M;R),
\label{eq:h=maxH}
\end{align}
the maximum rank of an isotropic submodule.  
We study the structure of (maximal) isotropic submodules for finite connected sums and direct products of manifolds. In particular, we show (Theorems~\ref{theor:conn_sum} and~\ref{theor:dir_prod}) that
\begin{align}
&\begin{alignedat}1
\h(M_1\connsum M_2;R)&=\h(M_1;R)+\h(M_2;R),\\ %
\h(M_1 \times M_2;R)&=\set{1}\cup \h(M_1;R)\cup \h(M_2;R), %
\end{alignedat}
\label{eq:intro-sum}
\shortintertext{and thus}
&\begin{alignedat}1
\hphantom{\h}\makebox[0pt][r]{$h$}(M_1 \connsum M_2;R)&=h(M_1;R)+h(M_2;R),\\
\hphantom{\h}\makebox[0pt][r]{$h$}(M_1 \times M_2;R)&=\max\set{h(M_1;R), h(M_2;R)}
\end{alignedat}
\label{eq:intro-prod}
\end{align}
under certain conditions and with certain exceptions described in the corresponding theorems (here the sum of sets is understood element-wise).

Isotropy index bounds
the co-rank $\b$ of the fundamental group, i.e., the maximum rank of a free homomorphic image of $\pi_1(M)$~\cite{Gelb10} and, obviously, is bounded by the first Betti number $b_1(M)=\rk H_1(M;\ZSet)$:
\begin{align}
\b\le h(M;\ZSet)\le b_1(M).
\label{eq:intro-b'<h<b_1}
\end{align}
In~\cite{Meln3}, 
for a field $F$
upper and lower bounds on $h(M;F)$ were given in terms of Betti numbers; see Proposition~\ref{prop:h-bounds}. 
Using~\eqref{eq:intro-sum} and~\eqref{eq:intro-prod}, for a given $R$ we 
describe all possible sets $\h(M;R)$ (Proposition~\ref{prop:hS})
and all possible values of $h(M,R)$ with different $M$ in terms of $b_1(M;R)=\rk H_1(M;R)$ (Theorem~\ref{theor:h,b}), 
as well as
extend~\eqref{eq:intro-b'<h<b_1} to fields of characteristic zero and show that in this case these bounds are exact
(Proposition~\ref{prop:b'<h<b1}).

The notion of isotropy has been studied in the context of algebraic geometry.
For instance,
isotropic subspace theorems by Catanese~\cite{Catanese} and Bauer~\cite{Bauer} establish relations between isotropic subspaces of $H^1(M;\CSet)$ for a smooth quasi-projective variety $M$ and certain irrational pencils. These theorems have been studied in~\cite{Dimca,Dimca-Su}.

The isotropy index has numerous applications to the topological study of manifolds and foliations.
As we show, isotropy for manifolds has a clear geometric meaning: 
(maximal) isotropic subgroups of $H^1(M;\ZSet)$ of rank $k$
correspond to (maximal) systems of $k$ homologically independent, homologically non-intersecting closed orientable codimension-one submanifolds, $h(M;\ZSet)$ being the maximum number of such submanifolds (Theorem~\ref{theor:h(M)_geometry}).
While this geometric meaning is defined for $R=\ZSet$,
we show that the relevant aspects of isotropy coincide for $R=\ZSet$ and $R=\QSet$
(Lemma~\ref{lem:h(G)=h(L)}), which 
enables the use of simpler, vector space-based techniques in geometric applications of isotropy.

Consider a foliation defined on $M$ by a Morse form $\w$, i.e., a closed one-form that is locally the differential of a Morse function. 
Such foliations have important applications in modern physics, for example, in supergravity~\cite{Bab-Laz-conf,Bab-Laz}.
A Morse form foliation defines a decomposition of $M$ into a finite number $m(\w)$ of minimal components and a finite number $M(\w)$ of maximal components, i.e., connected components of the union of compact leaves, which are cylinders over a compact leaf. 
These two numbers are bounded by $h(M;\ZSet)$:
$
M(\w)+m(\w)\le h(M;\ZSet)+|\Sing\w|-1,
$
where $\Sing\w$ is the singular set, which is finite~\cite{Gelb09}.
In homological terms, for the number $c(\w)$ of homologically independent compact leaves 
it holds
$
c(\w)+m(\w)\le h(M;\ZSet)
$~\cite{Gelb09}.

A sufficient condition of existence of a minimal component has been given in~\cite{Meln2} in terms of $\rk\w$, the rank of the group of the periods: if
$
\rk\w > h(M;\ZSet),
$
then the foliation has a minimal component, i.e., $m(\w)\ge1$.
Also, 
in case of strong inequality in the upper bound in~\eqref{eq:intro-b'<h<b_1}, i.e., if 
$
h(M;\ZSet)<b_1(M),
$
the foliation of a Morse form in general position has a minimal component~\cite{Gelb08}.

If the subgroup $H_\w\subseteq H_{n-1}(M)$ generated by the homology classes of all compact leaves of a Morse form foliation is maximal isotropic, then the foliation has no minimal components, i.e., $m(\w)=0$~\cite{Meln4}. 
Subgroups of $H_{n-1}(M)$ are related, by Poincar\'e duality, with those of $H^1(M;\ZSet)$.
In particular, if the homology classes of some compact leaves of a Morse form foliation generate a maximal isotropic subgroup, then $m(\w)=0$~\cite{Meln4}. 
Obviously, this is the case when the foliation has $h(M;\ZSet)$ homologically independent compact leaves. 
However, if $M=M_1\times M_2$, then in some cases our results
allow us to conclude that $m(\w)=0$
by examening only one leaf (Examples~\ref{ex:M22xS1-foliation} and~\ref{ex:H(M2gxM2g)}).

Isotropic submodules of $L$ can be defined for arbitrary bilinear map $\phi\from L\times L\to V$, where $L,V$ are finitely generated groups or finite-dimensional vector spaces. This gives the corresponding notions of $\h(\phi)$ and the isotropy index~$h(\phi)$ as in~\eqref{eq:h=maxH}. 
In order to establish our main results~\eqref{eq:intro-sum} and~\eqref{eq:intro-prod}, we study the behavior of isotropic submodules for skew-symmetric maps under operations of extension of scalars (Proposition~\ref{prop:h_over-diff-rings}) and direct sum $L_1\oplus L_2$ (Lemmas~\ref{lem:sum} and~\ref{lem:prod}).

By Poincar\'e duality, $\h(M;\ZSet)$ and $h(M;\ZSet)$ can be equivalently defined in terms of the intersection map 
$
\circ\colon H_1(M)\times H_1(M)\to\ZSet
$
instead of the cup product~\eqref{eq:intro-cupprod}.
For a closed one-form $\w$ on $M$,
the isotropy index $h(\w)$ is defined by the restriction $\circ|_{\ker[\w]\times \ker[\w]}$ of the intersection map 
to the group $\ker[\w]\subseteq H_1(M)$,
where $[\w]$ is the integration map.
This notion has been extensively used to study the structure of Morse form foliations on closed orientable surfaces $M^2_g$ of genus~$g$.
For example,
$
c(\w)\le h(\w);
$
if the foliation has no minimal components, then
$
h(\w)=g
$~\cite{Gelb13}.
For so-called weakly generic forms, 
$
m(\w)\ge g-\frac12k(\w)-h(\w)
$,
where $k(\w)$ is the number of singularities surrounded by a minimal component~\cite{Gelb13a}.
\textforidiot{$k(\w)=|\Sing\w\cap\bigcup_{i=1}^{m(\w)}\interior{\overline{C^{min}_i}}|$.}
Since $h(\w)\le h(M;\ZSet)$, these inequalities hold also for $h(M;\ZSet)$.

\medskip

The paper is organized as follows.
In Section~\ref{sec_def-spaces}, we give basic facts on isotropy in finite-dimensional vector spaces and finitely generated abelian groups.
In Section~\ref{sec_def_manif}, we introduce the isotropy index for manifolds and consider its properties and geometric meaning.
In Section~\ref{sec_conn-sum}, we calculate the isotropy index of the connected sum of two manifolds.
In Section~\ref{sec_dir-prod}, we calculate the isotropy index of the direct product of two manifolds and describe the possible sets $\h(M;R)$. 
In Section~\ref{sec_realization:h,b}, we completely characterize the relation between $h(M;R)$ and~$b_1(M;R)$.
Finally, in Section~\ref{sec_corank} we consider the relations between $h(M;R)$ and~$b_1'(M)$.

\section{\label{sec_def-spaces}Isotropy index for vector spaces and abelian groups}

In this section, we will define the isotropy index and discuss how it changes from groups to vector fields or between vector fields with different scalars.

We will deal with finite-dimensional vector spaces and finitely generated abelian groups. 
To avoid duplication of terminology, such as ``any subspace or subgroup'' or ``its dimension or rank'',
we will use terminology from $R$-modules, where $R$ will be a field or $\ZSet$, correspondingly. 
In particular, \emph{submodule} will stand for subspace or subgroup; 
\emph{rank} will stand for the dimension of a space or the rank of a group.

\subsection{Definitions}%

Let $L$, $V$ be  finitely generated  $R$-modules 
and $\varphi\colon L\times L\to V$ a bilinear map; $R$ is a field or $R=\ZSet$.

\begin{definition}\label{def:isotropic-space}
A submodule $H\subseteq L$ is called {\em isotropic} under the map $\varphi$
if $\varphi|_{H\times H}=0$, i.e., $\varphi (l_1,l_2)=0$ for any $l_1,l_2\in H$. 
\end{definition}

If $R$ is a field, $R$-modules $L$ and $V$ are  finite-dimensional vector spaces, so we deal with {\em isotropic subspaces}; if $R=\ZSet$, then $R$-modules  $L,V$ are finitely generated abelian groups, so we have {\em isotropic subgroups}.

Since $L$ is Noetherian, every isotropic submodule is contained in some maximal isotropic submodule, not necessarily unique.
\textforidiot{Since any finitely generated right module over a right Noetherian ring is a Noetherian module.}
Denote by $\h(\varphi)$ the set of ranks of maximal isotropic submodules under the map $\varphi$:
$$
\h(\varphi) =\set{\rk H\mid H\mbox{ is a maximal isotropic submodule of } L}.
$$
Obviously, $\h(\phi)$ is a finite set of non-negative integers such that
\begin{align}
0\notin\h(\phi)\text{\quad or\quad}\h(\phi)=\set{0}.
\label{eq:0notinHorH=0}
\end{align}
Proposition~\ref{prop:hS} below shows that these are the only restrictions on $\h(\phi)$.

\begin{definition}
The {\em isotropy index} $h(\varphi)$ 
is the maximum rank of the isotropic submodules of $L$: $$h(\varphi)=\max\h(\varphi).$$
\end{definition}

\begin{example}%
Consider the skew-symmetric map $\varphi\colon\RSet^3\times\RSet^3\to\RSet^3$, $\varphi(x,y)=[[x,y],l]$, where $l$ is a fixed vector and $[\;,\;]$ is the vector product.
For any vector $x\not\perp l$, for example $x=l$, the subspace $L_1=\langle x\rangle$, $\dim L_1=1$, is maximal isotropic, and so is
$L_2=l^{\perp}$, $\dim L_2=2$.
Thus 
$\h(\varphi)=\set{1,2}$, and
$h(\varphi)=2$.%
\end{example}

For skew-symmetric maps, usually $h(\varphi)\ge1$ and thus $0\notin\h(\phi)$:

\begin{lemma}\label{lem:h(L)=0}
Let $\phi$ be skew-symmetric.
Then $h(\varphi)=0$, i.e., $\h=\set{0}$, iff either
\begin{itemize}
\renewcommand\labelitemi{--}
\itemsep0em
\item $L=0$ or
\item $L=R$, 
$\har R=2$, and $\varphi\not\equiv 0$.
\end{itemize}
\end{lemma}

\begin{proof}
Let $h(\varphi)=0$, then $\phi(l,l)\ne0$ for any $0\ne l\in L$.
Unless $L=0$, for a skew-symmetric map this implies $\har R=2$.
Suppose $\rk L\ge 2$. Consider independent $l_1,l_2\in L$; $\varphi(l_i,l_i)=1$. Then for $l=l_1+l_2\ne0$, we have $\varphi(l,l)=0$, a contradiction.
\end{proof}

The {\em kernel} of a bilinear map $\phi\from L\times L\to V$ is 
$$
\ker\varphi=\set{l\in L\mid\varphi (l,l')=0\;\mbox{for any}\; l'\in L}.
$$
Obviously, $\ker\varphi$ is an isotropic submodule; moreover, any maximal isotropic submodule contains $\ker\varphi$, so $h(\varphi)\ge\dim\ker\varphi$. 

\subsection{Isotropy index for different coefficients}%

Generally speaking, 
the iso\-tropy index
depends on the coefficients. 
Namely, let $L$, $V$ be finitely generated abelian groups and\/ ${\phi}\colon L\times L\to V$ a skew-symmetric bilinear map. 
For a field $F$, consider the corresponding vector spaces 
\begin{align*}
L_F&=F\otimes L,\\
V_F&=F\otimes V 
\end{align*}
and the induced skew-symmetric bilinear map
$$
{\phi}_F\colon L_F\times L_F\to V_F,\quad \phi_F(\al_1\otimes x_1, \al_2\otimes x_2)=\al_1\al_2\otimes\phi(x_1, x_2).
$$
The isotropy index depends on the field $F$, and generally $h({\phi})\ne h({\phi}_F)$:

\begin{example}
Consider $\phi\from\ZSet\times\ZSet\to\ZSet_2=\ZSet/2\,\ZSet$, $\phi(1,1)=1$. 
It has an isotropic subgroup is $2\,\ZSet$, thus $h(\phi)=1$. 
Similarly, for $F=\QSet$, we have $\phi_\QSet\from\ZSet\times\ZSet\to0$, thus again $h(\phi)=1$, in accordance with~\eqref{eq:hZ=hQ} below.
However, for $F=\ZSet_2$, we have $L_F=\ZSet_2\otimes\ZSet=\ZSet_2$, $V_F=\ZSet_2\otimes\ZSet_2=\ZSet_2$, so ${\phi}_{\ZSet_2}\colon \ZSet_2\times\ZSet_2\to\ZSet_2$, $\phi(1,1)=1$;
obviously, ${\phi}_{\ZSet_2}^{-1}(0)=0$
and thus $h({\phi}_{\ZSet_2})=0$:
\begin{align*}
h(\phi_{\ZSet_2})<h(\phi)=h(\phi_\QSet).
\end{align*}
On the other hand, consider ${\phi}\from\ZSet^2\times\ZSet^2\to\ZSet$ defined by the matrix
{\scriptsize
$\begin{pmatrix*}[r]
0&k\\
\!-k&0
\end{pmatrix*}
\!,
$
}
$k\ge2$.
Then $h(\phi)=h_\QSet(\phi)=1$, but $\phi_{\ZSet_p}\from\ZSet_p^2\times\ZSet_p^2\to0$ (thus $h(\phi_{\ZSet_p})=2$) iff $p\mid k$. So for $p\mid k$ and $q\nmid k$ we have:
\begin{align*}
h(\phi)=h(\phi_\QSet)=h(\phi_{\ZSet_q})<h(\phi_{\ZSet_p}).
\end{align*}
\end{example}

However, extension of scalars for vector spaces does not decrease $h(\phi_F)$:

\begin{lemma}\label{lem:h_ext-of-scalars}
Let $L_F,V_F$ be finite-dimensional vector spaces over a filed~$F$ and ${\phi_F}\colon L_F\times L_F\to V_F$ be a skew-symmetric bilinear map.
Let $F'$ be a field, $F\subseteq F'$, 
\begin{align*}
L_{F'}&=F'\otimes_F L,\\
V_{F'}&=F'\otimes_F V
\end{align*}
vector spaces obtained from $L_F$, $V_F$ by extension of scalars, and $\phi_{F'}$ the induced map:
\begin{align}
{\phi_{F'}}\colon L_{F'}\times L_{F'}\to V_{F'},\quad \phi_{F'}(\al'_1\otimes x_1, \al'_2\otimes x_2)=\al'_1\al'_2\otimes\phi_F(x_1, x_2).
\label{eq:phi_F'}
\end{align}
Then 
$$
h(\phi_F)\le h(\phi_{F'}).
$$
\end{lemma}

\begin{proof}
Consider an isotropic subspace $H_F\subseteq L_F$, $\dim H_F=h(\phi_F)=k$. 
A basis $\left\langle e_1,\dots,e_k\right\rangle=H_F$ can be be extended to a basis for $L_F$; thus $H_F= F^k$.
Extension of scalars from $F$ to $F'$ gives $H_{F'}={F'}\otimes_F H_F={F'}\otimes_F F^k={(F')}^k$, so $\dim H_{F'}=\dim H_F$.
By~\eqref{eq:phi_F'}, the subspace $H_{F'}$ is isotropic, i.e., $k=\dim H_{F'}\le h(\phi_{F'})$. 
We obtain $h(\phi_F)\le h(\phi_{F'})$.
\end{proof}

A stronger fact holds for groups and $\QSet$:

\begin{lemma}\label{lem:h(G)=h(L)}
Let $L$, $V$ be finitely generated abelian groups
and\/ ${\phi}\colon L\times L\to V$ be a skew-symmetric bilinear map. Denote by 
\begin{alignat*}{2}
L_\ZSet&=L,\qquad &L_\QSet&=\QSet\otimes L,\\
V_\ZSet&=V,\qquad &V_\QSet&=\QSet\otimes V
\end{alignat*}
the corresponding $\ZSet$- and $\QSet$-modules.
Let
$$
{\phiq}\colon L_\QSet\times L_\QSet\to V_\QSet,\quad \phiq(q_1\otimes x_1, q_2\otimes x_2)=q_1q_2\otimes\phi(x_1, x_2)
$$
be the induced skew-symmetric bilinear map. 
Then 
\begin{enumerate}
\renewcommand{\labelenumi}{$\mathrm{\theenumi}$}
\renewcommand{\theenumi}{(\roman{enumi})}
\item\label{item:G->L}
for every maximal isotropic subgroup $H\subseteq L$, the subspace $\hq=\QSet\otimes H\subseteq L_\QSet$ is maximal isotropic;
\item\label{item:L->G}
for every maximal isotropic subspace $\hq\subseteq L_\QSet$, there is a maximal isotropic subgroup $H\subseteq L$ such that $\hq=\QSet\otimes H$.
\end{enumerate}
In particular, \textforidiot{since $\rk H=\dim\QSet\otimes H$}
\begin{align}
\begin{aligned}
\h({\phi})&=\h({\phiq}),\\
h({\phi})&=h({\phiq}).
\end{aligned}
\label{eq:hZ=hQ}
\end{align}
\end{lemma}

\begin{proof}
\ref{item:G->L}
Let $H\subseteq L$ be a maximal isotropic subgroup, $H=\langle h_1,\dots,h_n\rangle$.
Then $\hq=\QSet\otimes H\subseteq L_\QSet$ is an isotropic subspace, $\dim\hq=\rk H$.
\textforidiot{This is by Lemma~\ref{lem:dimQxG=rkG}.}
Consider $0\ne q\otimes x\in L_\QSet$ such that $\phiq(q\otimes x,\hq)=0$, i.e., all $\phi(x, h_i)\in V_T$, the torsion subgroup.
\textforidiot{This is by Lemma~\ref{lem:qxg=0=>q_in_T}.}
Then for some $k\ne0$,
\textforidiot{$k=\prod k_i$, where $k_ih_i=0$, $k_i\ne0$,}
we have $\phi(kx, h_i)=0$. 
Since $H$ is maximal, $kx\in H$ and thus 
\textforidiot{$1\otimes kx\in H_\QSet$, which gives}
$q\otimes x\in H_\QSet$.
Therefore, $H_\QSet$ is maximal.

\ref{item:L->G}
Let $H_\QSet\subseteq L_\QSet$ be a maximal isotropic subspace, $\hq=\langle q_1\otimes h_1,\dots,q_n\otimes h_n\rangle$, a basis.
Consider $H'=\langle h_1,\dots,h_n\rangle$. 
Then all $\phi(h_i,h_j)\in V_T$; thus for some $k\ne0$, all $\phi(kh_i, kh_j)=0$.
We obtain $H_\QSet=\QSet\otimes H''$ for an isotropic subgroup $H''=kH'=\set{kx\mid x\in H'}$, $k\ne0$.
\textforidiot{This is by Lemma~\ref{QxG=QxkG}.}

Consider a maximal isotropic subgroup $H\supseteq H''$. 
\textforidiot{This is by Lemma~\ref{H_in_maxH}.}
For any $x\in H$, we have $\phi(x, H'')=0$ and thus $\phiq(1\otimes x,\hq)=0$. 
Since $H_\QSet$ is maximal, $1\otimes x\in H_\QSet$.
We obtain $H_\QSet=\QSet\otimes H$\textforidiot{ and thus $\rk H=\dim H_\QSet$}.
\textforidiot{This is by Lemma~\ref{lem:dimQxG=rkG}.}
\end{proof}

Lemma~\ref{lem:h(G)=h(L)} allows formulating Lemma~\ref{lem:h_ext-of-scalars} for fields or $\ZSet$:

\begin{proposition}\label{prop:h_over-diff-rings}
Let $L_R,V_R$ be finitely generated $R$-modules, $R$ being a field or $\ZSet$, and $\phi_R\colon L_R\times L_R\to V_R$ be a skew-symmetric bilinear map.
Let $R'$ be a field,
$R\subseteq R'$,
\begin{align*}
L_{R'}&=R'\otimes_R L,\\
V_{R'}&=R'\otimes_R V
\end{align*}
modules obtained by extension of scalars, and 
$
\phi_{R'}\colon L_{R'}\times L_{R'}\to V_{R'}
$
the induced map.
Then 
\begin{align}
h(\phi_R)\le h(\phi_{R'}).
\label{eq:h<h}
\end{align}
\end{proposition}

In particular, in addition to extension of scalars of vector spaces, \eqref{eq:h<h} holds for a group and a corresponding vector space over $F$, $\har F=0$, since $\ZSet\subset\QSet\subseteq F$.

\section {\label{sec_def_manif}Isotropy index for manifolds}

In this section, we introduce maximal isotropic subgroups (subspaces) of the first cohomology group (space) and the isotropy index for manifolds and discuss their geometric meaning and properties.

\subsection {Definitions}%

Let $M$ be a smooth closed orientable $n$-dimensional manifold. %
Consider the cup product
\begin{align*}
{\smile}\colon &H^1(M;R)\times H^1(M;R) \to H^2(M;R),
\end{align*}
where $R=\ZSet$ or $R$ is a field. It is a skew-symmetric bilinear map, and $H^k(M;R)$ are finitely generated $R$-modules; in case of a field, $H^k(M;R)$ are vector spaces.

\begin{definition}
A submodule $H\subseteq H^1(M;R)$ is called {\em isotropic} if it is isotropic under $\smile$ in the sense of Definition~\ref{def:isotropic-space}, i.e., if
the restriction of the cup-product 
to $H\times H$ is zero: ${\smile}|_{H\times H}=0$. 
\end{definition}

Accordingly, we denote by $\h(M;R)$ the set of ranks of maximal isotropic submodules:
$$
\h(M;R) =\set{\dim H\mid H\mbox{\, is a maximal isotropic submodule of } H^1(M;R)};
$$
Proposition~\ref{prop:hS} below shows that~\eqref{eq:0notinHorH=0} is still the only restriction on $\h(M;R)$, i.e., that
almost any set of non-negative integers is $\h(M;R)$ for some manifold~$M$.

The {\em isotropy index} $$h(M;R)=\max\h(M;R)$$ is the maximum rank of iso\-tro\-pic submodules of $H^1(M;R)$.

Lemma~\ref{lem:h(G)=h(L)} allows us to work interchangeably with $H^1(M;\ZSet)$ and $H^1(M;\QSet)$:

\begin{lemma}
\label{lem:h=h_Q}
For a smooth closed orientable manifold $M$,
there exists a maximal isotropic subgroup $H\subseteq H^1(M;\ZSet)$, $\rk H=k$, iff there exists a maximal isotropic subspace $H_Q\subseteq H^1(M;\QSet)$, $\dim H_Q=k$, i.e.,
\begin{align*}
\h(M;\ZSet)&=\h(M;\QSet),\\
h(M;\ZSet)&=h(M;\QSet).
\end{align*}
\end{lemma}

\subsection {Geometric meaning of $h(M;\ZSet)$}%

The notions of $\h(M;\ZSet)$ and $h(M;\ZSet)$ have a clear geometric meaning, which can be characterized as follows:

\begin{definition}\label{def:isotropic-system}
An {\em isotropic system} on a manifold $M$, $\dim M\ge2$, is a set of homologically non-in\-ter\-sec\-ting 
homologically independent smooth closed orientable connected co\-di\-mension-one submanifolds $X_1,\dots,X_k\subset M$, intersecting transversely:
\begin{align}
[X_i\cap X_j]=0,
\label{eq:[XY]=0}
\end{align}
$i\ne j$; $i,j=1,\dots,k$. 
\end{definition}

The requirement~\eqref{eq:[XY]=0} cannot be simplified to $X_i\cap X_j=\emptyset$, since on some manifolds there exist submanifolds with non-empty, but homologically trivial, intersection:

\begin{example}\label{ex:Heisenberg-detailed}
On the Heisenberg 3-nilmanifold, any two 
homologically independent smooth closed orientable 2-submanifolds have non-empty, but homologically trivial, intersection.
This will be shown as Example~\ref{ex:Heisenberg} below; here we only give a graphical illustration.

The Heisenberg 3-nilmanifold $H^3$ is a $T^2$-bundle over the circle $S^1$, 
with the monodromy being
a Dehn twist 
$f\from T^2\to T^2$, defined as the quotient space
$$
H^3=\frac{[0,1]\times T^2}{(1,x)\sim (0,f(x))},\quad f=\begin{pmatrix} 1 & 0 \\ 1 & 1\end{pmatrix}.
$$
For the basis cycles $a$, $c$ of the torus $T^2$, we have
\begin{align*}
f_*(a)&=a+c;\\
f_*(c)&=c;
\end{align*}
see \figurename~\ref{fig:Heisengerg}({\it a}).
The cycle $c$ is homologically trivial, being realized by the boundary of a 2-submanifold (torus without a disk) shown in \figurename~\ref{fig:Heisengerg}({\it a}).
However, 
for the two submanifolds $T_i=T^2$ shown in \figurename~\ref{fig:Heisengerg}({\it b}), we have $[T_1\cap T_2]=c$.

\begin{figure}[ht]
\centering
\setlength\unitlength{0.2ex}
\vspace{10\unitlength}%
\includegraphics[width=100\unitlength]{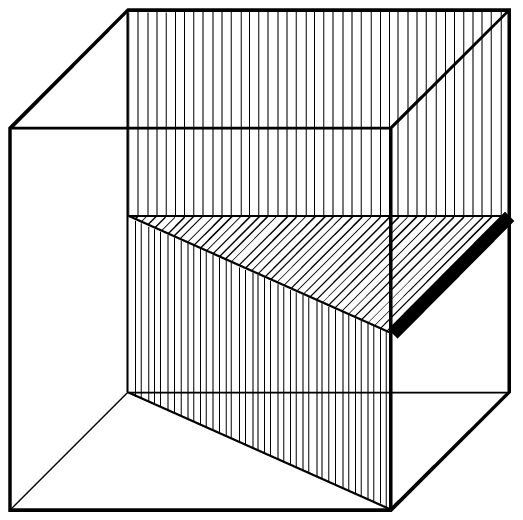}%
\begin{picture}(0,0)(100,0) %
\put(60,103){$a$}
\put(-8,35){$b$}
\put(87,37){$c$}
\put(25,14){\rotatebox{-25}{$a+c$}}
\put(50,-15){\makebox[0pt][c]{({\it a})}}
\end{picture}
\hspace{25\unitlength}
\includegraphics[width=100\unitlength]{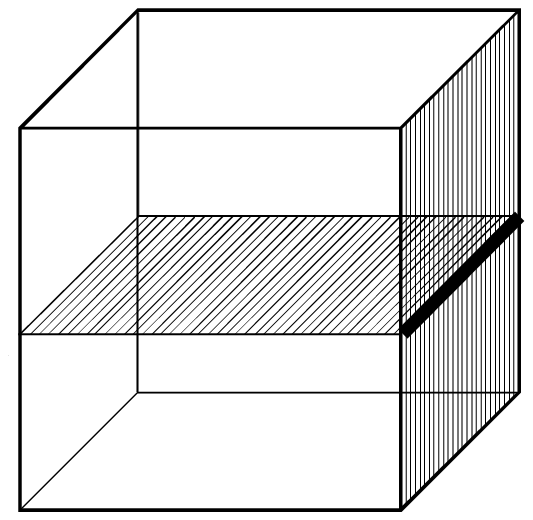}%
\begin{picture}(0,0)(100,0)
\put(45,63){$T_1$}
\put(103,37){$T_2$}
\put(50,-15){\makebox[0pt][c]{({\it b})}}
\end{picture}%
\vspace{15\unitlength}%
\caption{\label{fig:Heisengerg}
The Heisenberg nilmanifold $H^3$, represented as a $T^2$-bundle over the circle $S^1$. 
The circle is shown as the vertical line $b$; $T^2$ is shown as a horizontal square with the sides $a$ and $c$, the opposite sides of the square being identified.
The top is identified with the bottom with the Dehn twist: $a\sim a+c$, $c\sim c$; all four vertical lines are identified.
(a) The curve realizing $c$ is the boundary of a 2-submanifold shown as the hatched rectangle, triangle, and another rectangle; thus $c=0$.
The triangle forms a disk with two holes, which are glued to the cylinder formed by the two rectangles; the resulting figure is a torus with a disk removed, whose boundary realizes~$c$.
(b)~The two tori $T_1$, $T_2$ intersect by a curve realizing $c$, thus $[T_1\cap T_2]=0$. They cannot be made non-intersecting; see 
Example~\protect\ref{ex:Heisenberg}.}
\end{figure}

An algebraic model of $H^3$ can be given as follows: consider the 3-dimensional Heisenberg group over a ring $R$,
$$
H(R)=\set{\begin{pmatrix} 1 & x & z \\ 0 & 1 & y \\ 0 & 0 & 1\end{pmatrix}\mid x,y,z\in R};
$$
then the Heisenberg nilmanifold $H^3=H(\RSet)/H(\ZSet)$ is the quotient of the real Heisenberg group by the discrete Heisenberg subgroup.
It is a compact orientable connected 3-manifold with Nil geometry.
\end{example}

Definition~\ref{def:isotropic-system} implies the cardinality of an isotropic system $k\le b_1(M)$, the Betti number; thus each isotropic system is contained in a maximal isotropic system.

The following theorem relies on the fact that
homology classes $z\in H_{n-1}(M)$ can be realized by smooth closed orientable connected co\-di\-mension-one submanifolds.

\begin{theorem}\label{theor:h(M)_geometry}
Let $M$ be a smooth closed orientable connected manifold, $\dim M\ge2$,
and $D\colon H^1(M;\ZSet)\to H_{n-1}(M)$ be a Poincar\'e duality map. 
Then:
\begin{enumerate}
\renewcommand{\labelenumi}{$\mathrm{\theenumi}$}
\renewcommand{\theenumi}{(\roman{enumi})}
\item%
Let $\set{X_i}$ be a (maximal) isotropic system. 
Then $\set{D^{-1}[X_i]}$ form a basis of a (maximal) isotropic subgroup $H\subseteq H^1(M;\ZSet)$.
\item%
Let $\set{x_i}$ be a basis of a (maximal) isotropic subgroup $H\subseteq H^1(M;\ZSet)$. 
Then $\set{Dx_i}$ can be realized by submanifolds $X_i$ that form
a (maximal) isotropic system.
\end{enumerate}
In particular, 
\begin{itemize}
\renewcommand\labelitemi{--}
\item
$\h(M;\ZSet)=\set{k\mid\,X_1,\dots,X_k\subset M\text{ is a maximal isotropic system}}$;
\item
the isotropy index $h(M;\ZSet)$ is the maximum 
cardinality of
an isotropic system
of submanifolds of $M$.
\end{itemize}
\end{theorem}

\begin{proof}
Consider an isotropic 
subgroup $H\subset H^1(M;\ZSet)$, $\rk H=k$. 
Since $H^1(M;\ZSet)$ is torsion-free, it has a basis,
$H=\left\langle u_1,\dots,u_k\right\rangle$. 
The cup product 
$$
{\smile}\colon H^1(M;\ZSet)\times H^1(M;\ZSet)\to H^2(M;\ZSet)
$$ 
is dual to the homology classes intersection map
$$
{\circ}\colon  H_{n-1}(M)\times H_{n-1}(M)\to H_{n-2}(M);
$$
namely,  
$$
D(u_i\smile u_j)=Du_i\circ Du_j,
$$
where $D$ is a Poincar\'e duality map. 

Realize the cycles $Du_i\in H_{n-1}(M)$ by suitable submanifolds $X_i\subset M$, $Du_i=[X_i]$, choosing them to intersect transversely. Then
\begin{align*}
\pm[X_i\cap X_j]=[X_i]\circ[X_j]=Du_i\circ Du_j=D(u_i\smile u_j),
\end{align*}
where the sign depends on the choice of orientation in $X_i$ and $X_j$.
Since $H$ is isotropic, all $u_i\smile u_j=0$; 
thus
$[X_i\cap X_j]=0$ for any $i\ne j$. Since $u_i$ are independent, so are $[X_i]$. 
If $H$ is maximal then so is this system, because expanding it would, by duality, expand $H$.

Similarly, given a (maximal) system of $k$ such submanifolds $X_i\subset M$, the group $H=\langle D^{-1}[X_i]\rangle\subseteq H^1(M;\ZSet)$ is a (maximal) isotropic subgroup, $\rk H=k$.
\end{proof}

Examples~\ref{ex:Heisenberg} and~\ref{ex:KT} below show that the homological interpretation of the non-intersection requirement~\ref{eq:[XY]=0} is important for Theorem~\ref{theor:h(M)_geometry}:
some manifolds have fewer non-intersecting submanifolds, $X_i\cap X_j=\emptyset$, with the properties listed in Definition~\ref{def:isotropic-system}, than homologically non-intersecting such submanifolds, $[X_i\cap X_j]=0$. For discussion of systems with $X_i\cap X_j=\emptyset$, see Section~\ref{sec_corank}.

\subsection {Properties and examples}%

Recall that the Betti number $b_k(M;R)=\rk H_k(M;R)$; by definition, $b_k(M)=b_k(M;\ZSet)$. By the  universal coefficient theorem, 
if $R$ is a field with $\har R=0$, then $b_k(M;R)=b_k(M)$. 
Since $H^k(M;R)=H_k(M;R)=0$ for $k>\dim M$, the following statements apply to $S^1$ and a point~$*$.

\begin{lemma}\label{lem:h(M)=0}
Let $M$ be a smooth closed orientable manifold; $R=\ZSet$ or $R$ be a field. Then
\begin{align}
1\le h(M;R)\le b_1(M;R),
\label{eq:1<h<b}
\end{align} 
except that
\begin{align} 
h(M;R)=0
\label{eq:h=0}
\end{align} 
iff any of the following conditions holds:
\begin{itemize}
\renewcommand\labelitemi{--}
\itemsep0em
\item %
    $b_1(M;R)=0$, or
\item %
    $b_1(M;R)=1$, 
    $\har R=2$, and the cup product ${\smile}\not\equiv 0$.
\end{itemize}
\end{lemma}

Theorem~\ref{theor:h,b} below states that this lemma gives the only relation between $h(M;R)$ and $b_1(M;R)$.

\begin{proof}
By definition, we have $h(M;R)\le\rk H^1(M;R)$. 
For $\ZSet$, by Poincar\'e duality, $H^1(M;\ZSet)\cong H_{n-1}(M)$, so $\rk H^1(M;\ZSet)=b_{n-1}(M)=b_1(M)$.
For a field $F$,
$H^1(M;F)\cong H_1(M;F)$, so $\dim H^1(M;F)=\dim H_1(M;F)=b_1(M;F)$.
Since the cup product is skew-symmetric,~\eqref{eq:h=0} is given by Lemma~\ref{lem:h(L)=0}.
\end{proof}

\begin{example}\label{ex:RP3}
Consider $M=\RSet P^3$; it is orientable. Its cohomology ring is 
\begin{align*}
H^*(\RSet P^3;\ZSet_2)\cong\frac{\ZSet_2[\al]}{(\al^4)},
\end{align*}
where $|\al|=1$. 
Thus each $H^i(\RSet P^3;\ZSet_2)$
is a free $\ZSet_2$-module 
with generator $\al^i$,
i.e.,
$H^i(\RSet P^3;\ZSet_2)=\ZSet_2$.
We have $b_i(\RSet P^3;\ZSet_2)=1$, and for $\al\in H^1(\RSet P^3;\ZSet_2)$ it holds $\al\smile\al\ne0$, i.e., $h(\RSet P^3;{\ZSet_2})=0$.
\end{example}

Obviously, $h(M;R)=b_1(M;R)$ iff  ${\smile}\equiv 0$.
Since $\ker{\smile}\subseteq H^1(M;R)$ is an isotropic submodule, $h(M;R)\ge\dim\ker{\smile}$.
There are, though, better estimates:

\begin{proposition}\label{prop:h-bounds}
Let $M$ be a smooth closed orientable manifold
and %
$k=\dim\ker{\smile}$. For $R=\ZSet$ or $R$ being a field, with the exception specified below, we have:
\begin{enumerate}
\item It holds
\begin{align}
\frac{b_1(M;R)+k\,b_2(M;R)}{b_2(M;R)+1}\le h(M;R)\le\frac{b_1(M;R)\,b_2(M;R)+k}{b_2(M;R)+1};
\label{eq:b+bk<h<bb+k}
\end{align}
in particular, if\/ $b_2(M;R)=1$, then
\begin{align}
h(M;R) = \frac 12(b_1(M;R) + k).
\label{eq:h=1/2}
\end{align}
\item
If\/ $\smile$ is surjective, then
\begin{align}
h(M;R)\le k+\frac 12+\sqrt{\left(b_1(M;R)-k-\frac 12\right)^2-2\,b_2(M;R)}.
\label{h=sqrt}
\end{align}
\end{enumerate}
As an exception, if\/ 
\begin{align}
\left\{
\begin{array}l
\har R=2,\\
b_1(M;R)=1,\\
k=0,
\end{array}
\right.
\label{eq:exception}
\end{align}
then $h(M;R)=0$ and
of\/~\eqref{eq:b+bk<h<bb+k}--\eqref{h=sqrt}, only the upper bound in~\eqref{eq:b+bk<h<bb+k} holds.
\end{proposition}

\begin{proof}
If $h(M;R)\ne0$,
for a field
this has been shown in~\cite{Meln3}; for $\ZSet$ it follows from Lemma~\ref{lem:h=h_Q}. 

If $h(M;R)=0$, then by Lemma~\ref{lem:h(M)=0} either $b_1(M;R)=0$ or~\eqref{eq:exception} holds.
In the former case,~\eqref{eq:b+bk<h<bb+k}--\eqref{h=sqrt} happen to hold; in~\eqref{h=sqrt}, we have $b_2(M;R)=0$.
In the latter case,~\eqref{eq:h=1/2} and the lower bound in~\eqref{eq:b+bk<h<bb+k} do not hold, while in~\eqref{h=sqrt} we have $b_2(M;R)=1$ and the square root does not exist.
\end{proof}

The exception is illustrated by Example~\ref{ex:RP3}. In fact,~\eqref{eq:h=1/2} and the lower bound in~\eqref{eq:b+bk<h<bb+k} would not need an exception if we took the floor function of the corresponding expressions, which in all other cases except~\eqref{eq:exception} happen to be integer anyway.

\begin{example}
For a closed orientable surface of genus $g$, \eqref{eq:h=1/2} gives $h(M^2_g;R)=g$;
for $n$-torus, $n\ge2$,~\eqref{h=sqrt} gives $h(T^n;R)=1$. 
Both cases do not fall under exception~\eqref{eq:exception}, since $b_1(M^2_g)=2g\ne1$ and $b_1(T^n;R)=n\ne1$.
\end{example}

\begin{example}
\label{ex:H(T)}
Thus, $\h(T^n;R)=\set{1}$.
Indeed,~\eqref{eq:0notinHorH=0} gives $$h(M;R)=1\quad\text{iff}\quad\cal H(M;R)=\set{1}.$$ 
\end{example}

The following example shows a non-singleton $\h(M;R)$:

\begin{example}\label{ex:M22xS1}
$\h(M^2_2\times S^1;\ZSet)=\set{1,2}$. This is seen in \figurename~\ref{fig:M22xS1}, but can also be formally proved by Theorem~\ref{theor:dir_prod} below.
\end{example}

\begin{figure}[ht]
\centering
    \includegraphics[angle=-90,width=10cm]{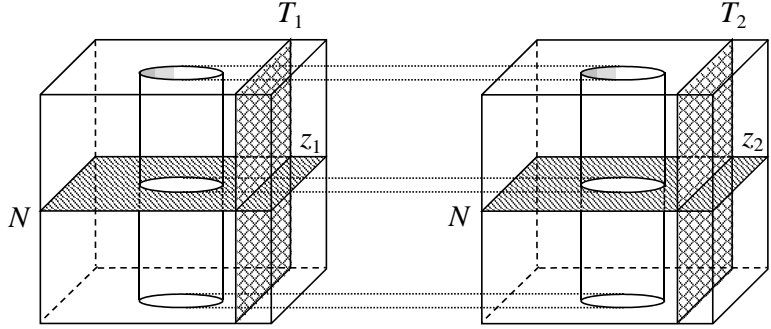}
    \caption{\label{fig:M22xS1}Maximal systems of non-intersecting submanifolds of different cardinality: $\h(M^2_2\times S^1;\ZSet)=\set{1,2}$. 
    The double torus $M^2_2=T^2\connsum T^2$, labeled by $N$, is shown as two horizontal squares, sides of each one being pairwise identified, glued together by a removed circle in the center; $S^1$ is shown as vertical lines.
    The two maximal systems of homologically non-intersecting, homologically independent submanifolds are $\set{N}$, a double torus, and $\set{T_1,T_2}$, two tori $z_i\times S^1$, 
    $z_1$ and $z_2$ being two homologically independent cycles in the $M^2_2$.}
\end{figure}

\section{\label{sec_conn-sum}Isotropy index of the connected sum of manifolds}

For sets, we denote $A+B=\set{a+b\mid a\in A, b\in B}$. 

\begin{lemma}\label{lem:sum}
Let\/ $L_i$, $V_i$, $i=1,2$, be finite-dimensional vector spaces over a filed $F$ and\/
$\varphi_i\colon L_i\times L_i\to V_i$ be bilinear skew-symmetric maps. 
Denote
\begin{align*}
L&=L_1\oplus L_2,\\
V&=V_1\oplus V_2, 
\end{align*}
and let $\varphi\colon L\times L\to V$ be a bilinear skew-symmetric map such that
\begin{align*}
\varphi|_{L_i\times L_i}&=\varphi_i,\\
\varphi|_{L_1\times L_2}&=0;
\end{align*}
i.e., $\varphi$ is
defined as component-wise sum of $\varphi_i$:
\begin{align}
\varphi(x,y)=\underbrace{\varphi_1(x_1,y_1)}_{\in V_1}+\underbrace{\varphi_2(x_2,y_2)}_{\in V_2},
\label{phi=phi+phi}
\end{align}
where $x_i,y_i\in L_i$ are projections.
Then:
\begin{enumerate}
\item
A subspace $H\subseteq L$ is maximal isotropic iff\/
\begin{align}
H=H_1\oplus H_2, 
\label{eq:H=H+H}
\end{align}
where $H_i\subseteq L_i$ are maximal isotropic under $\varphi_i$.
\item
The set of dimensions of maximal isotropic subspaces
$$
\h({\phi})=\h({\phi}_1)+\h({\phi}_2).
$$
\item
The isotropy index $$h({\phi})=h({\phi}_1)+h({\phi}_2).$$
\end{enumerate}
\end{lemma}

Note that these conclusions do not necessarily hold for isotropic subspaces that are not maximal. 
For example, each 1-dimensional subspace $\langle x\rangle$, $x\in L\setminus(L_1\cup L_2)$, is isotropic, \textforidiot{since $\har F\ne 2$} but~\eqref{eq:H=H+H} does not hold for it.

\begin{proof}
By~\eqref{phi=phi+phi}, if $H_i\in L_i$ are isotropic, then $H=H_1\oplus H_2$ is isotropic.

($\Rightarrow$) 
Let $H\subseteq L$ be a maximal isotropic subspace. 
Consider $x,y\in H$. 
By~\eqref{phi=phi+phi}, $\varphi(x_1,y_1)=-\varphi(x_2,y_2) \in V_1\cap V_2=0$, i.e., both projections $p_i(H)\subseteq L_i$ are isotropic. 
Let $H_i \supseteq p_i(H)$ be maximal isotropic subspaces of $L_i$. Then $H'=H_1\oplus H_2$ is isotropic, and since $H\subseteq H'$ is maximal, $H=H'$.

($\Leftarrow$) 
Conversely, let $H_i \subseteq L_i$ be maximal isotropic subspaces; then
$H=H_1\oplus H_2$ is isotropic. 
Consider $x=x_1+x_2 \in L\setminus H$, i.e., say, $x_1\notin H_1$. 
Since $H_1$ is maximal isotropic in $L_1$, there exists $y=y_1\in H_1\subseteq H$ such that $\varphi(x_1,y_1)\ne 0$. 
By~\eqref{phi=phi+phi}, $\varphi(x,y)=\varphi(x_1,y_1)+0\ne 0$. Thus $H$ is maximal. 
\end{proof}

\begin{theorem}\label{theor:conn_sum}
Let $M_1$, $M_2$ be connected closed orientable manifolds with $\dim M_i\ge 2$, and $R$ be a field or $R=\ZSet$.
Then for the connected sum $M_1\connsum M_2$:
\begin{enumerate}
\item
A submodule $H\subseteq H^1(M_1\connsum M_2;R)$ 
is maximal isotropic iff $$H=H_1\oplus H_2,$$ where $H_i\subseteq H^1(M_i;R)$ are maximal isotropic submodules.
\item
The set of ranks of maximal isotropic submodules
$$
\h(M_1\connsum M_2;R)=\h(M_1;R)+\h(M_2;R).
$$
\item
The isotropy index of the connected sum
$$
h(M_1 \connsum M_2;R)=h(M_1;R)+h(M_2;R).
$$
\end{enumerate}
\end{theorem}

\begin{proof}
Let $R$ be a field. 
Denote $L=H^1(M_1\connsum M_2;R)$ and $L_i=H^1(M_i;R)$, $i=1,2$.
Since $\dim M_i\ge2$, the  Mayer-Vietoris sequence gives $L=L_1\oplus L_2$.
The additive structure is given by the induced maps of the inclusions; the cup product translates into component-wise product: 
\begin{align*}%
x\smile y=(x_1\smile y_1)+(x_2\smile y_2),
\end{align*}
where $x,y\in L$ and $x_i,y_i\in L_i$ are projections.
Then, for fields, Lemma~\ref{lem:sum} gives the result.
Now for $R=\ZSet$, the result follows from Lemma~\ref{lem:h(G)=h(L)}.
\end{proof}

\begin{example}\label{ex:H(M^2_g)}
By Theorem~\ref{theor:conn_sum} and given Example~\ref{ex:H(T)}, or by~\eqref{eq:h=0} if $g=0$,
for a closed orientable surface of genus $g$ it holds
$$
\h(M_g^2;R)=\h(\connsum_{i=1}^g T^2;R)=\sum_{i=1}^g \h(T^2;R)=\set{g}.
$$
\end{example}

\begin{example}
Consider $M=M^2_2\times S^1$ from Example~\ref{ex:M22xS1} with $\h(M;\ZSet)=\set{1,2}$; see \figurename~\ref{fig:M22xS1}. Then
$$\h(M\connsum M;\ZSet)=\set{2,3,4}$$
and $h(M\connsum M;\ZSet)=4$.
\end{example}

\begin{example}
Consider $M=M^2_a\times M^2_b$, surfaces of genus $a$ and $b$, $1\le a\le b$. Theorem~\ref{theor:dir_prod} below gives
$\h(M;R)=\set{1,a,b}.$ Therefore,
$$\h(M\connsum M;R)=\set{2,a+1,b+1,a+b,2\,a,2\,b}$$
and $h(M\connsum M;R)=2\,b$.
\end{example}

\section{\label{sec_dir-prod}Isotropy index of the direct product of manifolds}

\begin{lemma}\label{lem:xy=yx}
Let $L$ be a vector space and $x,y,u,v\in L$; $x,y\ne 0$.
Then $x\otimes v=u\otimes y$ implies $u=ax$, $v=ay$ for some $a$.
\end{lemma}

\begin{proof}
Coordinate-wise, we have
\begin{align}
x_iv_j=u_iy_j
\label{eq:xv=uy}
\end{align}
for all $i,j$. For those $i,j$ for which $x_i,y_j\ne0$, this gives $$\frac{u_i}{x_i}=\frac{v_j}{y_j}=a_{ij}.$$ Since $a_{ij}$ does not depend on $i$ or $j$, all $a_{ij}=a$. We obtain $u_i=ax_i$ if $x_i\ne0$ and $v_j=ay_j$ if $y_j\ne0$. 
If $x_i=0$,~\eqref{eq:xv=uy} gives $0=u_iy_j$ for all $j$, thus $u_i=0$, and similarly $y_j=0$ implies $v_j=0$.
\end{proof}

While in Lemma~\ref{lem:sum} we had $\varphi|_{L_1\times L_2}=0$, now consider 
$\im\varphi|_{L_1\times L_2}$
as large as possible: %

\begin{lemma}\label{lem:prod}
Let\/ $L_i$, $V_i$, $i=1,2$, be final-dimensional vector spaces over a field $F$ and\/
$\varphi_i\colon L_i\times L_i\to V_i$ be bilinear skew-symmetric maps. 
Denote
\begin{align*}
L&=L_1\oplus L_2,\\
V&=V_1\oplus V_2\oplus V_3,
\end{align*}
where $V_3=L_1\otimes L_2$, and let $\varphi\colon L\times L\to V$ be a bilinear skew-symmetric map such that
\begin{gather}
\begin{aligned}
\varphi|_{L_i\times L_i}&=\varphi_i,\\
\varphi|_{L_1\times L_2}&=\otimes;
\end{aligned}
\notag
\shortintertext{i.e.,}
\varphi(x,y)=\underbrace{\varphi_1(x_1,y_1)}_{\in V_1}+\underbrace{\varphi_2(x_2,y_2)}_{\in V_2}+\underbrace{x_1\otimes y_2-y_1\otimes x_2}_{\in V_3},
\label{eq:phi=phi+phi+phi}
\end{gather}
where $x_i,y_i\in L_i$ are projections.
Then:
\begin{enumerate}
\item
A subspace $H\subseteq L$ is isotropic iff $$\dim H=1 \textforidiot{\har F=2?} \textrm{ or }H=H_i,$$
where $H_i\subseteq L_i$ is isotropic under $\varphi_i$, for $i=1\textrm{ or }2$.
\item\label{item:x2}
The set of dimensions of maximal isotropic subspaces
$$
\h({\phi})=\set{1}\cup \h({\phi}_1)\cup \h({\phi}_2)
$$
except that\/ $\h({\phi})=\h({\phi}_i)$ if\/ $h(\phi_j)=0$, i.e., if either
\begin{itemize}
\renewcommand\labelitemi{--}
\itemsep0em
\item $L_j=0$ or 
\item $L_j=F$, $\har F=2$, and $\varphi_j\not\equiv0$.
\end{itemize}
\item
The isotropy index $$h({\phi})=\max\set{h({\phi}_1),h({\phi}_2)}.$$
\end{enumerate}
\end{lemma}

Note that in contrast to Lemma~\ref{lem:sum}, the first conclusion does not require $H$ to be maximal.

\begin{proof}
Let $H$ be an isotropic subspace.
We will show that if both projections $p_i(H)\ne 0$, then $\dim H=1$.
Consider $x\in H$ such that both projections $x_i\ne 0$. Let $y\in H$.
Since $H$ is isotropic and the three components of~\eqref{eq:phi=phi+phi+phi} are independent, we have
$$
\varphi_1(x_1,y_1)=\varphi_2(x_2,y_2)=x_1\otimes y_2-y_1\otimes x_2=0.
$$
By Lemma~\ref{lem:xy=yx}, $y\in\langle x\rangle$.
The conditions for $h(\phi_j)=0$ in item~\eqref{item:x2} are given by Lemma~\ref{lem:h(L)=0}.
\end{proof}

\begin{theorem}\label{theor:dir_prod}
Let $M_1$, $M_2$ be connected closed  manifolds and $R$ be a field or $R=\ZSet$.
Then for the direct product $M_1 \times M_2$:
\begin{enumerate}
\item
A submodule $H\subseteq H^1(M_1 \times M_2;R)$ is isotropic iff 
$$\rk H=1\textrm{ or }H=H_i,$$
where $H_i\subseteq H^1(M_i;R)$ is isotropic for $M_i$, $i=1\text{ or }2$.
\item
The set of ranks of maximal isotropic submodules %
$$
\h(M_1 \times M_2;R)=\set{1}\cup \h(M_1;R)\cup \h(M_2;R)
$$
except that\/ $\h(M_1 \times M_2;R)=\h(M_i;R)$ if\/ $h(M_j;R)=0$, i.e., if either
\begin{itemize}
\renewcommand\labelitemi{--}
\itemsep0em
\item $b_1(M_j;R)=0$, the Betti number, or 
\item $b_1(M_j;R)=1$, $\har R=2$, and\/ ${\smile}\not\equiv0$.
\end{itemize}
\item
The isotropy index of the direct product 
$$
h(M_1 \times M_2;R)=\max\set{h(M_1;R), h(M_2;R)}.
$$
\end{enumerate}
\end{theorem}

\begin{proof}
Let $R$ be a field. 
Denote
\begin{align*}
L_i&=H^1(M_i,R),\quad i=1,2,&L&=H^1(M_1 \times M_2,R),\\
V_i&=H^2(M_i,R),\quad i=1,2,&V&=H^2(M_1 \times M_2,R).
\end{align*}
By the K\"unneth formula,
\begin{align*}
L&=L_1\oplus L_2,\\
V&=V_1\oplus V_2\oplus V_3,
\shortintertext{where}
V_3&=L_1\otimes L_2.
\end{align*}
By construction, $L_i\smile L_i\subseteq V_i$ for $i=1,2$; $L_1\smile L_2\subseteq V_3$, and~\eqref{eq:phi=phi+phi+phi} holds
\textforidiot{by definition of cup product}
for the cup-products in $M_1\times M_2$ and $M_i$, respectively.
Lemma~\ref{lem:prod} gives the result for fields and Lemma~\ref{lem:h(G)=h(L)} for $R=\ZSet$.
\end{proof}

Example~\ref{ex:Heisenberg-detailed} shows that in Theorem~\ref{theor:dir_prod}, the direct product cannot be replaced by an arbitrary fiber bundle.

\begin{example}\label{ex:h_R(T^n)}
By Lemma~\ref{lem:h(M)=0},
$h(S^1;R)=1$, so for a torus 
$T^n=\opTimes_{i=1}^n S^1$,
we have $h(T^n;R)=1$.
\end{example}

\begin{example}\label{ex:S1*SN}
By Lemma~\ref{lem:h(M)=0},
$h(S^n;R)=0$, $n\ge 2$, and %
$h(S^1;R)=1$, so $h(S^n\times S^1;R)=1$.
\end{example}

\begin{proposition}\label{prop:hS}
For any non-empty finite set\/ $S\subset\ZSet^*$ of non-negative integers and for $R=\ZSet$ or $R$ being a field, 
$S=\h(M;R)$ for some smooth closed orientable connected manifold\/ $M$ iff\/ $S=\set{0}$ or $0\notin S$. 
\end{proposition}

\begin{proof}
\textforidiot{($\Rightarrow$) The conditions on $\h(M;R)$ follow from the fact that $\set{0}$ is a subset of any submodule.}
If $S=\set{g}$, then $S=\h(M^2_g;R)$, a surface of genus $g$; see Example~\ref{ex:H(M^2_g)}.
Let now $S=\set{s_1,\dots,s_N}$, $N\ge2$. By the condition, $m=\min S\ge1$.
Consider
\begin{alignat*}2
M_1&=M^2_{s_1-m+1}\times\cdots\times M^2_{s_N-m+1},&\quad&\dim M_1=2N,\\
M_2&=M^2_{m-1}\times S^{2N-2},&&\dim M_2=2N.
\end{alignat*}
By Theorems~\ref{theor:conn_sum} and~\ref{theor:dir_prod}, we obtain $\h(M_1\connsum M_2;R)=S$.
\end{proof}

Another application of Theorem~\ref{theor:dir_prod} can be found 
in study of the topology of foliations defined by Morse forms. 
It is known that if the subgroup of $H_{n-1}(M)$ generated by the homology classes of all compact leaves of the foliation is maximal isotropic, then the foliation has no minimal components~\cite{Meln4}. 
This condition obviously holds true when the foliation has $h(M;\ZSet)$ homologically independent compact leaves.
However, if $M=M_1\times M_2$, in some cases Theorem~\ref{theor:dir_prod} allows to conclude that the foliation has no minimal components by considering only one leaf: 

\begin{example}\label{ex:M22xS1-foliation}
As has been mentioned in Example~\ref{ex:M22xS1},
$\h(M^2_2\times S^1;\ZSet)=\set{1,2}$; see \figurename~\ref{fig:M22xS1}.
Even though $h(M^2_2\times S^1,\ZSet)=2$, if a Morse form foliation has the submanifold $N=M^2_2$ as a leaf, then it has no minimal components.
In contrast, nothing can be said about a form that has $T_1=T^2$ as a leaf, because the system $\set{T_1}$ is not maximal.
\end{example}

\begin{example}\label{ex:H(M2gxM2g)}
$\h(M^2_a\times M^2_b;\ZSet)=\set{1,a,b}$, $a,b\ge1$.
Now consider a cycle $z$ that winds around the $M_1=M^2_a$ and also around the $M_2=M^2_b$, that is, $z=z_1+z_2$, $0\ne z_i\in H^1(M_i,\ZSet)$. 
If a Morse form foliation has a leaf dual to $z$, 
then it has no minimal components.
\end{example}

\section{\label{sec_realization:h,b}Isotropy index and the first Betti number}

By definition of the isotropy index, 
\begin{gather*}
h(M;R)\le b_1(M;R);
\shortintertext{for example:}
\begin{alignedat}2
h(S^1;R)&=1,\quad&b_1(S^1;R)&=1;\\
h(M^2_g;R)&=g,\quad&b_1(M^2_g;R)&=2g;\\
h(T^n;R)&=1,\quad&b_1(T^n;R)&=n.
\end{alignedat}
\end{gather*}
The only relation between $h(M;R)$, $b_1(M;R)$, and $R$ is given by Lemma~\ref{lem:h(M)=0}; in particular, any gap between $h(M;R)$ and $b_1(M;R)$ is possible for a given $R$:

\begin{theorem}\label{theor:h,b}
Let\/ $h,b\in\ZSet$, and $R$ be a field or $R=\ZSet$. 
There exists a connected smooth closed orientable manifold\/ $M$ with $h(M;R)=h$ and\/ $b_1(M;R)=b$ iff any of the following conditions holds:
\begin{itemize}
\renewcommand\labelitemi{--}
\itemsep0em
\item $1\le h\le b$, or
\item $h=b=0$, or
\item $h=0$, $b=1$, and\/ $\har R=2$.
\end{itemize}
\end{theorem}

\begin{proof}
For $h=b=0$, consider $M=S^n$. 
For $h=0$ and $b=1$ with $\har R=2$, consider $M=\RSet P^3$; see Example~\ref{ex:RP3}.
Let now $1\le h\le b$. Choose $m_i\ge1$ such that 
\begin{align}
\sum_{i=1}^{h} m_i=b.\label{sum=b}
\end{align}
For large enough $n$ such that $n-m_i\ge2$ for all $i$, consider an $n$-manifold
$$
M=\opConnsum_{i=1}^h\left(T^{m_i}\times S^{n-{m_i}}\right).
$$
By Theorem~\ref{theor:dir_prod}, for each summand $M_i=T^{m_i}\times S^{n-{m_i}}$, we have $h(M_i;R)=1$, while $b_1(M_i;R)=m_i$.  
Then by Theorem~\ref{theor:conn_sum},
\begin{align*}
h(M;R)&=\sum_{i=1}^h h(M_i;R)=\sum_{i=1}^h1 =h,\\
b_1(M;R)&=\sum_{i=1}^h b_1(M_i,R)=\sum_{i=1}^h m_i=b.
\qedhere
\end{align*}
\end{proof}

Note that the construction used in the proof requires $n=\dim M\ge2+\lceil \frac b{h}\rceil$, the ceiling here being the smallest possible value for $\max\set{m_i}$ under~\eqref{sum=b}.
This condition is not restrictive when $h=b$, leading to 
$n\ge3$; it is not very restrictive when $\frac b2\le h<b$, leading to $n\ge4$, etc.;
for $\frac{b}{k+1}\le h<\frac{b}{k}$, $k=1,\dots,b-1$, we need $n\ge k+3$. This 
requires high dimension when $h\ll b$.

For $\ZSet_2$, however, 
$n=3$ is enough:

\begin{proposition}
For $R=\ZSet_2$, the manifold in Theorem~\ref{theor:h,b} can be chosen with any given $\dim M\ge3$.
\end{proposition}

\begin{proof}
Let $R=\ZSet_2$ and $\dim M=3$. For $h=b=0$, consider $M^3=S^3$.
Let now $b\ge1$. Consider 
\begin{align}
M^3=\left(\opConnsum_{i=1}^h\left(S^1\times S^2\right)\right)\opConnsum\left(\opConnsum_{i=1}^{b-h}\RSet P^3\right).
\label{eq:M3}
\end{align}
Example~\ref{ex:S1*SN} shows that $h(S^1\times S^2;\ZSet_2)=1$, thus Theorem~\ref{theor:conn_sum} implies
\begin{alignat*}4
h(M^3;\ZSet_2)  &=\sum_{i=1}^h h(S^1\times S^2;\ZSet_2)&&{}+\sum_{i=1}^{b-h} h(\RSet P^3;\ZSet_2) &&{}=\sum_{i=1}^h1+\sum_{i=1}^{b-h}0 &&{}=h,\\
b_1(M^3;\ZSet_2)&=\sum_{i=1}^h b_1(S^1\times S^2;\ZSet_2)&&{}+\sum_{i=1}^h b_1(\RSet P^3;\ZSet_2) &&{}=\sum_{i=1}^h1+\sum_{i=1}^{b-h}1 &&{}=b.
\end{alignat*}
This trivially generalizes to $\dim \ge 5$ as
\begin{align}
M^n=M^3\times S^{n-3}
\label{eq:M3timesS}
\end{align} 

Let now $\dim M=4$. For 
$1\le h<b$, we use~\eqref{eq:M3timesS} with one summand less in~\eqref{eq:M3}, namely,
$$
M^3=\left(\opConnsum_{i=1}^h\left(S^1\times S^2\right)\right)\opConnsum\left(\opConnsum_{i=1}^{b-h-1}\RSet P^3\right).
$$
For $h=b$, consider
$
M^4=\opConnsum_{i=1}^h(S^1\times S^3)
$.

Finally, for $h=0$, $b=1$, consider an Enriques surface $X$. 
\textforidiot{
http://mathoverflow.net/questions/209251/closed-orientable-4-manifold-with-h1m-bbb-z-2-bbb-z-2-and-non-zero-cup-pr
see also 
http://math.stackexchange.com/questions/1325568/closed-orientable-4-manifold-with-h1m-bbb-z-2-bbb-z-2-and-non-zero-cup-pr/1325764#1325764
}
Indeed,\footnote{Example contributed by a colleague who preferred not to be named.}
$X=K3/\sigma$, where $\sigma$ is an orientation-preserving fixed point-free involution; note that a $K3$ surface is simply connected. 
Then $H^1(X;\ZSet_2)=\Hom(\pi_1(X),\ZSet_2)=\ZSet_2$; thus $b_1(X;\ZSet_2)=1$.
For $0\ne x\in H^1(X;\ZSet)$, $x\smile x$ is a reduction (mod 2) of $\beta x$, where $\beta\from H^1(X;\ZSet_2)\to H^2(X;\ZSet)$ is the Bockstein homomorphism. 
Suppose $x\smile x=0$, i.e., for some $y\in H^2(X;\ZSet)$ we have $\beta x=2y$ with $\beta x\ne 0$ because $H^1(X;\ZSet)=0$. 
Since $\beta x$ is $2$-torsion, we obtain that $0\ne\pi^*y\in H^2(K3;\ZSet)$ is $4$-torsion, where $\pi$ is the quotient map, while the latter group is torsion-free. 
Thus $h(X;\ZSet)=0$.
\end{proof}

\section{\label{sec_corank}Isotropy index and the co-rank of the fundamental group}

In this section, we give a lower bound on $h(M;R)$ stronger than $1$ from~\eqref{eq:1<h<b}.

\begin{definition}
The {\em co-rank of the fundamental group} 
of a smooth closed connected manifold $M$ is 
the maximum rank of a free quotient group of $\pi_1(M)$; we denote it by $\b$.
\end{definition}

While $h(M;\ZSet)$ is the maximum number of homologically non-intersecting submanifolds $[X_i\cap X_j]=0$ (Theorem~\ref{theor:h(M)_geometry}), $\b$ 
strengthens the condition to
$X_i\cap X_j=\emptyset$:

\begin{theorem}[\textrm{\cite[Theorem 2.1]{Jaco72}}]\label{theor:b'(M)_geometry}
The co-rank of the fundamental group
$\b$ is the maximum number of non-intersecting homologically independent 
smooth closed orientable connected codimension-one submanifolds $X_i\subset M$:
$$X_i\cap X_j=\emptyset,$$ $i\ne j$, $i,j=1,\dots,\b$.
\end{theorem}

\begin{foridiot}
The condition for $X_i$ to be smooth closed orientable is important for deduce that $\b\le h(M)$.
Jaco's~\cite[Theorem 2.1]{Jaco72} is formulated for a compact $M$, possibly with $\partial M$; $X_i$ are also compact and $\partial X_i\subseteq\partial M$. 
Since our $M$ is closed, then $X_i$ are also closed.
These $X_i$ are two-sided; since $M$ is orientable, so are $X_i$.
These $X_i$ are not necessary smooth, but any homology class can be realized by a smooth $X_i$.
\end{foridiot}

Accordingly, properties of $\b$ closely resemble those of $h(M;\ZSet)$.
Similarly to~\eqref{eq:1<h<b}--\eqref{eq:h=0}, it holds~\cite{Gelb18}:
\begin{gather}
\b=0\textrm{\quad iff\quad}b_1(M)=0
\label{eq:b'=0-iff-b1=0}
\shortintertext{and otherwise}
1\le \b\le b_1(M);
\label{eq:1<b'<b}
\end{gather}
in particular, $h(M;\ZSet)=0$ iff $\b=0$. 
Exactly as in Theorems~\ref{theor:conn_sum} and~\ref{theor:dir_prod}, for the connected sum, $\dim M_i\ge 2$, except for non-orientable surfaces, and the direct product it holds
\begin{alignat*}2%
b_1'(M_1\connsum M_2)&=b_1'(M_1)+b_1'(M_2),&\quad&\text{\cite{Harvey}}\\
b_1'(M_1\times M_2)&=\max\set{b_1'(M_1), b_1'(M_2)}.&&\text{\cite{Gelb18}}
\end{alignat*}

\begin{example}\label{ex:b'=h}
Non-surprisingly, for many manifolds $\b=h(M;R)$:
\begin{itemize}
\renewcommand\labelitemi{--}
\item
For the closed orientable surface, $b_1'(M^2_g)=g$~\cite{Leininger} and $h(M^2_g;R)=g$~\cite{Meln3}; see Example~\ref{ex:H(M^2_g)}.
\item
For $n$-torus, $b_1'(T^n)=1$~\cite{Gelb17} and $h(T^n;R)=1$~\cite{Meln3}; see Example~\ref{ex:h_R(T^n)}. 
\item
For manifolds with quasi-K\"ahler and 1-formal fundamental group, for example, for compact K\"ahler manifolds, $\b=h(M;\CSet)$~\cite{Dimca-Pa-Su}.
\item
For 
$M=\opConnsum_{i=1}^h\left(T^{m_i}\times S^{n-{m_i}}\right)$ from Theorem~\ref{theor:h,b}, it holds $\b=h(M;R)$.
\end{itemize}
\end{example}

A non-trivial theorem from~\cite{Gelb18} implies
that~\eqref{eq:b'=0-iff-b1=0}--\eqref{eq:1<b'<b} represent the only relation between $\b$ and $b_1(M)$ for any given
$\dim M$.
The last item in Example~\ref{ex:b'=h} shows that the construction from Theorem~\ref{theor:h,b} gives an elementary proof of this fact for large enough $\dim M$: 

\begin{theorem}
Let\/ $b',b\in\ZSet$. There exists a connected smooth closed orientable manifold\/ $M$ 
with $\b=b'$ and\/ $b_1(M)=b$ iff either
\begin{alignat*}2
b'=b=0,&\quad\text{see~\eqref{eq:b'=0-iff-b1=0}, or}\\
1\le b'\le b,&\quad\text{see~\eqref{eq:1<b'<b}.}
\end{alignat*}
\end{theorem}

Comparing Theorems~\ref{theor:h(M)_geometry} and~\ref{theor:b'(M)_geometry} gives
\begin{align}
\b\le h(M;\ZSet);
\label{b'<h(M;Z)}
\end{align}
together with~\eqref{eq:1<h<b} this gives
a geometric proof of
lower and upper bounds on the isotropy index $h(M;\ZSet)$, which have been obtained indirectly in~\cite{Gelb10}. 
We extend this to fields of characteristic zero:

\begin{proposition}\label{prop:b'<h<b1}
Let $R=\ZSet$ or $R$ be a field, $\har R=0$. For 
the co-rank of the fundamental group $\b$,
the isotropy index $h(M;R)$, and 
the first Betti number $b_1(M)$ it holds
\begin{align}
\b\le h(M;R)\le b_1(M).
\label{h(M;R)<b1(M)}
\end{align}
\end{proposition}

\begin{proof}
By Proposition~\ref{prop:h_over-diff-rings}, for a field $F$ with $\har F=0$, we have $h(M;\ZSet)\le h(M;F)$. 
Equations~\eqref{b'<h(M;Z)} and~\eqref{eq:1<h<b} complete the proof:
\[ %
\b\le h(M;\ZSet)\le h(M;F)\le b_1(M).
\qedhere
\]
\end{proof}

Both bounds in~\eqref{h(M;R)<b1(M)} are exact (see Example~\ref{ex:b'=h} and Theorem~\ref{theor:h,b}); in particular, as we have shown, in many cases $\b$ is a very strong lower bound for $h(M)$.
However, both inequalities can also be strict:
\begin{example}\label{ex:Heisenberg}
Consider the Heisenberg nilmanifold $H^3$. Its fundamental group $\pi_1(H^3)$ is nilpotent,  so $b'_1(H^3)=1$.
Since  $H^1(H^3,\ZSet)=\ZSet^2$ with zero cup-product~\cite{Lambe-Pr}, we have 
$$
1=b'_1(H^3)<h(H^3;\ZSet)=b_1(H^3)=2.
$$
\end{example}

\begin{example}\label{ex:KT}
The Kodaira--Thurston nilmanifold $M=H^3\times S^1$ gives an example of 
$$
\b<h(M;\ZSet)<b_1(M).
$$
Indeed, the fundamental group $\pi_1(M)$ is nilpotent, so $\b=1$; by Theorem~\ref{theor:dir_prod} and given Example~\ref{ex:Heisenberg}, $h(M;\ZSet)=2$; and, obviously, $b_1(M)=3$.
\end{example}

\renewcommand\bibitem\bib

\end{document}